\documentclass[11pt]{amsart}
\usepackage{amssymb}
\usepackage{amsthm}
\usepackage{amsmath}
\usepackage[dvips]{epsfig}
\usepackage{graphicx}
\usepackage{hyperref}
\usepackage{color}
\usepackage{url}
\usepackage{lpic}
\usepackage{tikz}
\usepackage{caption}
\usepackage{subcaption}
\usepackage{epstopdf}
\usepackage{mathtools}
\usepackage[arrow, matrix, curve]{xy}
\usepackage{marginnote}

\usetikzlibrary{decorations.markings}

\makeatletter\@addtoreset {equation}{section}\makeatother

\newtheorem{theorem}{Theorem}

\newtheorem{lemma}{Lemma}
\theoremstyle{remark}
\newtheorem{remark}{Remark}
\theoremstyle{definition}
\newtheorem{definition}{Definition}
\theoremstyle{corollary}

{\begin{trivlist} \item[]{\em Proof }}%
{\hspace*{\fill}$\rule{.3\baselineskip}{.35\baselineskip}$\end{trivlist}}

\renewcommand{\O}{\mathcal{O}}

\DeclareMathOperator{\R}{\mathbb{R}}

\begin{document}

\title[Uniqueness and instability of the peaked periodic wave]{\bf Linear instability and uniqueness
of the peaked periodic wave in the reduced Ostrovsky equation}

\author{Anna Geyer}
\address[A. Geyer]{Delft Institute of Applied Mathematics, Faculty Electrical Engineering, Mathematics and
Computer Science, Delft University of Technology, Mekelweg 4, 2628 CD Delft, The Netherlands}
\email{A.Geyer@tudelft.nl}

\author{Dmitry E. Pelinovsky}
\address[D. Pelinovsky]{Department of Mathematics and Statistics, McMaster University,
Hamilton, Ontario, Canada, L8S 4K1}
\email{dmpeli@math.mcmaster.ca}
\address[D. Pelinovsky]{Department of Applied Mathematics, Nizhny Novgorod State Technical University,
24 Minin street, 603950 Nizhny Novgorod, Russia}


\begin{abstract}
Stability of the peaked periodic wave in the reduced Ostrovsky equation has remained an open problem for a  long time.
In order to solve this problem we obtain sharp bounds on the exponential growth of the $L^2$ norm of co-periodic perturbations to the
peaked periodic wave, from which it follows that the peaked periodic wave is linearly unstable.
We also prove that the peaked periodic wave with parabolic profile is the unique peaked wave
in the space of periodic $L^2$ functions with zero mean and a single minimum per period.
\end{abstract}

\date{\today}
\maketitle

\section{Introduction}

We address solutions of the Cauchy problem for the reduced Ostrovsky equation \cite{Ostrov}
written in the form
\begin{equation}
\label{redOst}
\left\{ \begin{array}{l} u_t + u u_x = \partial_x^{-1} u, \quad t > 0,\\
u|_{t=0} = u_0, \end{array} \right.
\end{equation}
where $u_0$ is a $2\pi$-periodic function with zero mean defined in the Sobolev space
$H^s_{\rm per}(-\pi,\pi)$ for some $s \geq 0$, which we simply write as $H^s_{\rm per}$.
We denote the subspace of $2\pi$-periodic functions with zero mean in $H^s_{\rm per}$
by $\dot{H}^s_{\rm per}$. The operator $\partial_x^{-1}: \dot{H}^s_{\rm per} \rightarrow \dot{H}^{s+1}_{\rm per}$
denotes the anti-derivative with zero mean, which can be defined using Fourier series.

The reduced Ostrovsky equation is also known under the names of Ostrovsky--Hunter and
Ostrovsky--Vakhnenko equation, due to contributions of Hunter \cite{Hunter}
and Vakhnenko \cite{vakh1}.

Local solutions to the Cauchy problem (\ref{redOst}) with $u_0 \in \dot{H}^s_{\rm per}$
exist for $s > \frac{3}{2}$ \cite{SSK10}, and we refer to \cite{LPS} for a discussion on how the well-posedness
in $H^s(\R)$ is extended to $\dot{H}^s_{\rm per}$.
For sufficiently large initial data, the local solutions break in finite time,
similar to the inviscid Burgers equation \cite{LPS}. However, if the initial data $u_0$
is suitably small, then the local solutions for $s = 3$ are continued for all times \cite{GH,GP}.
Weak bounded solutions with shock discontinuities were constructed in \cite{Coclite1,Coclite2}.
Weak solutions of the Cauchy problem (\ref{redOst}) as the limiting solution of the Cauchy problem for
the regularized Ostrovsky equation were considered in \cite{Coclite3}.

The reduced Ostrovsky equation with smooth solutions is completely integrable
as it can be reduced to the integrable Tzizeica equation by a coordinate transformation \cite{Manna}.
This property enables a construction of a bi-infinite set of conserved quantities in the time evolution \cite{Sakovich}
and the inverse scattering transform with the Riemann--Hilbert approach \cite{Shepelsky}.
Two integrable semi-discretizations of the reduced Ostrovsky equation have been obtained
by using bilinear forms \cite{Maruno}.

Stability of smooth and peaked periodic waves in the reduced Ostrovsky equation has been recently addressed
in a number of publications \cite{JP,GP17,Hakkaev1,Hakkaev2,Stefanov}. By using higher-order conserved quantities
the smooth small-amplitude periodic waves were shown in \cite{JP} to be unconstrained minimizers of a higher-order energy function.
This result holds for {\em subharmonic} perturbations, that is, perturbations whose period is  a multiple of the period of the smooth periodic waves. Since the higher-order conserved quantities
are well-defined in the space $\dot{H}^3_{\rm per}$, where global well-posedness has been proven \cite{GP},
it follows from the minimization properties that smooth small-amplitude
periodic waves are both spectrally and orbitally stable.
The minimization properties were confirmed numerically for smooth periodic waves of large amplitude all the way up to the limiting peaked wave of parabolic profile with maximal amplitude,  for which the numerical results were inconclusive \cite{JP}.

Spectral stability of smooth periodic waves with respect to {\em co-periodic} perturbations, that is,
perturbations with the same period as the period of the periodic wave, was shown in \cite{GP17}
by using the standard variational formulation of the periodic waves as critical points of energy subject to
fixed momentum. This result holds also for the generalized reduced Ostrovsky equation with power nonlinearity.
Independently, spectral stability of smooth periodic waves in the reduced
Ostrovsky equation was shown in \cite{Hakkaev2} by using a coordinate transformation of the spectral
stability problem to an eigenvalue problem studied earlier in \cite{Stefanov}.

Regarding the peaked periodic waves, some conflicting results were recently obtained.
In \cite{Hakkaev2}, the peaked wave with the parabolic profile was addressed and claimed
to be ``unstable in the absence of periodic boundary conditions".
A formal proof of this statement was obtained by constructing explicit solutions of the spectral stability problem for a positive
(unstable) eigenvalue. However, this construction violates the periodic boundary conditions on the perturbation and
hence does not provide an answer to the spectral stability question. In contrast, families of peaked  periodic waves of small amplitude,
which were previously unknown in the context of the reduced Ostrovsky equation,
were constructed in \cite{Hakkaev1} and these families were shown to be
spectrally stable with respect to co-periodic perturbations by using the same coordinate transformation
as in \cite{Stefanov}.

In this paper we give a simple and definite conclusion about existence, uniqueness and stability of peaked periodic waves in the reduced
Ostrovsky equation. This is the first time,
to the best of our knowledge, that  linear instability
of peaked periodic waves is proven by means of
semigroup theory and energy estimates.

The following theorem presents a summary of the main results of this paper.
See Definitions \ref{def-single-lobe}, \ref{def-linear-stability}
and Lemmas \ref{lemma-no}, \ref{lemma-full} for precise statements.

\begin{theorem}
\label{theorem-main}
\begin{enumerate}
\item[]
\item \emph{\bf Uniqueness:}
The peaked periodic wave $U_*$ with parabolic profile is the unique (up to spatial translations)
peaked travelling wave solution of the reduced Ostrovsky equation in $\dot L^2_{\rm per}$
having a single minimum per period. The solution is Lipschitz continuous and exists
in $\dot{H}^s_{\rm per}$ with $s < 3/2$. Moreover, the reduced Ostrovsky equation
does not admit any H\"{o}lder continuous solutions.

\item \emph{\bf Instability:}  The orbit generated by spatial translations
of the peaked periodic wave $U_*$ is \emph{linearly unstable} with respect to
perturbations in $X^1_{\rm per}$, where
\begin{equation}
\label{X-1-space}
X^1_{\rm per} := \{ v \in \dot{L}^2_{\rm per} : \; (c_* - U_*) v \in H^1_{\rm per} \}
\end{equation}
and $c_*$ is the wave speed of the periodic wave $U_*$.
\end{enumerate}
\end{theorem}

Part (1) of Theorem \ref{theorem-main} allows us to prove that the
families of peaked periodic small-amplitude waves constructed in \cite{Hakkaev1} do not satisfy the reduced
Ostrovsky equation, see Remark \ref{remark-small_amp_peakons}. Our analysis relies on Fourier theory and the existence of a first integral.
Indeed, the reduced Ostrovsky equation for smooth periodic waves can be rewritten as a second-order differential equation
with a conserved quantity. Although this equivalence can not be used when dealing with peaked periodic waves,
we can still use a first-order invariant of the second-order differential equation to analyze the behavior
of the smooth parts of the peaked periodic waves together with sharp estimates of the solution at the singularity,
see Remark \ref{rem_singularity} and Lemma \ref{lemma-no}.

Part (2) of Theorem \ref{theorem-main} gives a definite conclusion on linear instability
of the peaked periodic wave with parabolic profile with respect to co-periodic perturbations.
We do not make any claims regarding the spectral stability problem related
to the peaked periodic wave, see Remark \ref{rem-spectrum}. Instead, we prove linear instability
of the peaked periodic waves by obtaining sharp bounds on the exponential growth of the $L^2$ norm of the co-periodic
perturbations in the linearized time-evolution problem in $X^1_{\rm per}$, see Lemma \ref{lemma-full}.
Note that $\dot{H}^1_{\rm per}$ is continuously embedded into $X^1_{\rm per}$ but is not equivalent to $X^1_{\rm per}$, see Remark \ref{rem-embedding}.

It is interesting to compare peaked periodic waves in the reduced Ostrovsky equation with
peaked waves in other related nonlinear dispersive equations such as the Whitham equation
and the Camassa--Holm equation. The existence of smooth periodic travelling waves in the Whitham equation
has recently been established by \cite{Bruell18,Mats1,Ehrnstrom2013a,Mats2}, where it was shown
that the family of smooth periodic waves terminates at the highest, peaked wave, similarly to what
happens for the reduced Ostrovsky equation. It was shown numerically  in \cite{Sanford2014} that
smooth periodic waves of small amplitude are stable while smooth waves of large amplitude become unstable,
even before reaching the highest wave. This is different from the reduced Ostrovsky equation,
where all smooth periodic waves are stable even for large amplitudes up to the peaked wave, see \cite{JP},
whereas the peaked periodic wave is unstable. For the Camassa-Holm equation, both the smooth periodic waves
of all amplitudes and the limiting peaked periodic wave are stable, see \cite{Lenells2004b, Lenells2005d}
and the earlier result \cite{CS} on peakons. It is an open question to understand which precise mechanisms
govern these surprisingly different stability behaviours.

The paper is organized as follows. Section \ref{sec-2} contains the proof that the peaked wave with parabolic profile is
unique up to spatial translations in the space of functions in $\dot{L}^2_{\rm per}$ with a single minimum per period.
Section \ref{sec-3} gives the proof of linear instability of the peaked
periodic wave with respect to co-periodic perturbations.

\section{Peaked periodic wave}
\label{sec-2}

The periodic travelling waves in the reduced Ostrovsky equation are given by
$$
u(x,t) = U(x-ct),
$$
where $c$ is the wave speed and $U$ is a bounded
$2\pi$-periodic wave profile with zero mean. The wave profile $U$ is to be found
from the boundary-value problem
\begin{equation}
\label{ODE}
\left\{ \begin{array}{l} \left[ c - U(z) \right] U'(z) + (\partial_z^{-1} U)(z) = 0, \quad
\mbox{\rm for every } \; z \in (-\pi,\pi) \;\; \mbox{\rm such that } \; U(z) \neq c, \\
U(-\pi) = U(\pi), \quad \int_{-\pi}^{\pi} U(z) dz = 0, \end{array} \right.
\end{equation}
where $z = x-ct$ is the travelling wave coordinate. If $U \in \dot{L}^2_{\rm per}$,
then $\partial_z^{-1} U \in \dot{H}^1_{\rm per}$. By Sobolev's embedding, it follows that
$\partial_z^{-1} U \in C_{\rm per}$ so that the anti-derivative $\partial_z^{-1} U$
with zero mean can be expressed by the pointwise formula
\begin{equation}
\label{anti-derivative}
	(\partial_z^{-1} U)(z) = \int_{0}^z U(z') dz' - \frac{1}{2\pi}\int_{-\pi}^{\pi}\int_{0}^z U(z')dz'dz, \quad z \in [-\pi,\pi].
\end{equation}
    In what follows, we assume that $U$ is at least continuous on $[-\pi,\pi]$,
that is, we assume that $U \in C_{\rm per}$. For $\alpha \in (0,1)$, let
$C^{\alpha}_{\rm per}$ be the space of $\alpha$-H\"{o}lder $2\pi$-periodic continuous functions such that
\begin{equation}
\label{Holder}
|U(x) - U(y)| \leq K |x-y|^{\alpha}, \quad \mbox{\rm for all} \; x,y \in [-\pi,\pi],
\end{equation}
for  some $K\in \R$. We will adopt the following definition of single-lobe periodic waves.

\begin{definition}
\label{def-single-lobe}
We say that $U \in C_{\rm per}$ is a single-lobe periodic wave if
there exists $z_0 \in (-\pi,\pi)$ such that $U$ is non-increasing on $[-\pi,z_0]$
and non-decreasing on $[z_0,\pi]$.
\end{definition}

\begin{remark}
\label{remark-single-lobe}
Due to the condition $U(-\pi) = U(\pi)$ and the symmetry of the equation
$$
(c - U(z)) U'(z) + \int_{0}^z U(z') dz' - \frac{1}{2\pi}\int_{-\pi}^{\pi}\int_{0}^z U(z')dz'dz = 0
$$
with respect to the reflection $z \mapsto -z$,
the single-lobe periodic waves in Definition \ref{def-single-lobe} have even profile
$U$ with $z_0 = 0$. In this case, $(\partial_z^{-1} U)(z) = \int_{0}^z U(z') dz'$ is odd.
\end{remark}

A family of smooth $2\pi$-periodic waves to the boundary-value problem (\ref{ODE})
satisfying $U(z) < c$ for every $z \in [-\pi,\pi]$ was constructed in our previous work \cite{GP17}
in an open interval of the speed parameter $c$.
By Theorem 1(a) and Lemma 3 in \cite{GP17}, we have the following result.

\begin{lemma}
\label{lemma-smooth-wave}
There exists $c_* > 1$ such that for every $c \in (1,c_*)$,
the boundary-value problem (\ref{ODE}) admits a unique smooth periodic wave
in the sense of Definition \ref{def-single-lobe} with the profile $U \in \dot{H}^{\infty}_{\rm per}$
satisfying $U(z) < c$ for every $z \in [-\pi,\pi]$.
\end{lemma}

\begin{remark}
For the smooth periodic waves $U \in \dot{H}^{\infty}_{\rm per}$ to the boundary-value problem
(\ref{ODE}), the periodic boundary conditions are satisfied for all derivatives of $U$.
\end{remark}

At $c = c_*$, the periodic wave with parabolic profile has been known
since the original work of Ostrovsky \cite{Ostrov}. It is easy to check
that the boundary-value problem (\ref{ODE}) is satisfied
by $U(z) = (z^2-3c)/6$, whereas the zero mean condition is satisfied if $c = c_* \coloneqq \pi^2/9$. This yields
the exact expression for the peaked periodic wave with zero mean
\begin{equation}
\label{peaked-wave}
  U_*(z) \coloneqq \frac{3 z^2 - \pi^2}{18}, \quad z \in [-\pi,\pi],
\end{equation}
periodically continued beyond $[-\pi,\pi]$. Note that $U_*(\pm\pi)= \pi^2/9 = c_*$ and $\pm U'_*(\pm\pi)=\frac{\pi}{3}$.
 The peaked periodic wave (\ref{peaked-wave}) can be represented by
the Fourier cosine series
$$
U_*(z) = \sum_{n=1}^{\infty} \frac{2 (-1)^n}{3 n^2} \cos(nz),
$$
which is well defined in $\dot{H}^s_{\rm per}$ for $s < 3/2$.


\begin{remark}
\label{rem_angle}
The enclosed angle at the peak of the wave is steeper than the maximal $120^{\circ}$ angle of the Stokes wave of greatest height,
see \cite{Stokes1847a, Toland1978a}. Indeed, for peaked periodic waves of the reduced Ostrovsky equation the enclosed angle
is $\varphi=\pi-2\arctan (\pi/3)$, whereas for the Stokes wave it is $\varphi=\pi-2\pi/3$.
\end{remark}

\begin{remark}
\label{rem_singularity}
The peaked periodic wave (\ref{peaked-wave}) belongs to solutions of the boundary-value problem (\ref{ODE})
with profile $U_* \in \dot{H}^1_{\rm per}$ satisfying $U_*(z) < c$ for every $z \in (-\pi,\pi)$
and $U_*(\pm \pi) = c$. The first derivative of
$U_* \in \dot{H}^1_{\rm per}$ has a finite jump singularity across the end points $z = \pm \pi$.
More precisely, the profile $U_*$ is Lipschitz continuous at  $\pm\pi$, that is, there exist
constants $0<c_1<c_2$ such that
$$
c_1 |z-\pi| \leq |U_*(z)-c_*| \leq c_2 |z-\pi| \quad \mbox{\rm for} \;\; |z-\pi| \ll 1,
$$
which can be easily checked in view of the explicit expression \eqref{peaked-wave}.
\end{remark}

The next result states that the only single-lobe periodic wave with
profile $U \in C_{\rm per}$
satisfying the boundary-value problem (\ref{ODE}) and having a
singularity in the derivative at $z = \pm \pi$
is the peaked periodic wave $U_*$ given in (\ref{peaked-wave}).

\begin{lemma}
\label{lemma-no}
For every $c\in\R$, the boundary-value problem (\ref{ODE}) does not admit single-lobe periodic waves
in the sense of Definition \ref{def-single-lobe} which are $C^{\alpha}_{\rm per}$ with $\alpha \in [0,1)$.
The only periodic wave with a singularity in the derivative at $z = \pm \pi$ is the peaked wave
with parabolic profile (\ref{peaked-wave}), which exists for $c = c_* = \pi^2/9$ and is Lipschitz at the peak.
\end{lemma}

\begin{proof}
Let $U\in \dot L^2_{\rm per}\cap C_{\rm per}$ be a single-lobe periodic wave solution of
the boundary-value problem (\ref{ODE}). By Remark \ref{remark-single-lobe},
$U \in \dot{L}^2_{\rm per}$ is even, $\partial_z^{-1} U \in \dot{H}^1_{\rm per}$ is odd, and
$\partial_z^{-1} U$ is represented by the Fourier sine series
which converges absolutely and uniformly, so that $(\partial_z^{-1} U)(\pm \pi) = 0$.\medskip

$\bullet$ Let us first consider the case where $U(z) \neq c$ for every $z \in (-\pi,\pi)$ and  $U(\pm \pi) = c$.
Let  $\alpha \in (0,1)$.  We assume to the contrary that there exists a solution $U$ of
 the boundary-value problem (\ref{ODE}) with  $U \in C^{\alpha}_{\rm per}$. If $U \in C^{\alpha}_{\rm per}$ with $\alpha \in (0,1)$, then
$\partial_z^{-1} U \in C_{\rm per}^{1}$. Since $U(\pm \pi) = c$ and $(\partial_z^{-1} U)(\pm \pi)=0$  we find that $c - U(z) \sim (\pi - z)^{\alpha}$ and $(\partial_z^{-1} U)(z) \sim (\pi - z)$ at $z=\pm \pi$. Since $U$ satisfies the boundary value problem \eqref{ODE} we have that
\begin{equation}
U'(z) = -\frac{(\partial_z^{-1} U)(z)}{c-U(z)}, \quad z \in (-\pi,\pi)
\label{contradiction-ODE}
\end{equation}
which yields  $U'(z)\sim (\pi-z)^{1-\alpha}$ at $z=\pm \pi$. Equation (\ref{contradiction-ODE})
also implies that $U' \in C^1(-\pi,\pi)$ so we find  that $U' \in C_{\rm per}^{1-\alpha}$.
Since $1- \alpha \in (0,1)$ we conclude that $U \in C_{\rm per}^1$ in contradiction to
the assumption that $U \in C^{\alpha}_{\rm per}$ with $\alpha \in (0,1)$. The case $\alpha=0$,
which refers to solutions $U \in C_{\rm per}$ in view of Definition \ref{def-single-lobe},
can be proven in exactly the same way.

We now show that the only peaked periodic solution with peak at $U(\pm \pi) = c$ is the solution with the parabolic profile (\ref{peaked-wave}).
Since $U(z) \neq c$ for every $z \in (-\pi,\pi)$ and $U \in C^1(-\pi,\pi)$,
the first-order invariant
\begin{eqnarray}
\nonumber
E & = & \frac{1}{2} \left[ c - U(z) \right]^2 \left[ U'(z) \right]^2 + \frac{c}{2} U(z)^2 - \frac{1}{3} U(z)^3 \\
& = & \frac{1}{2} \left[(\partial_z^{-1} U)(z) \right]^2 + \frac{c}{2} U(z)^2 - \frac{1}{3} U(z)^3,
\label{energy-ODE}
\end{eqnarray}
holds for $z\in (-\pi,\pi)$.
Since $(\partial_z^{-1} U)(z)$ is continuous in $z=\pm\pi$ with $(\partial_z^{-1} U)(\pm \pi) = 0$,
$E$ is continuous and constant up to the boundary at $z=\pm \pi$ and we have $E |_{z = \pm\pi} = c^3/6=:E_c $.
For $c>0$, the level with $E = E_c$ (see the bold curve in Figure~\ref{fig-phase-plane}) gives rise to a  peaked periodic wave solution with parabolic profile $U(z)<c$. We claim that this peaked wave is exactly the solution \eqref{peaked-wave} with speed $c=c^*$.
Indeed, from the level set with $E=E_c$ we have that
$$
\tfrac{1}{2} (c-U)^2 (U'(z))^2=\tfrac{c^3}{6} - \tfrac{c}{2} U^2 + \tfrac{1}{3} U^3  = \tfrac{1}{6} (c-U)^2(c+2U)
$$
and hence
$$
U'(z) = \tfrac{1}{\sqrt{3}}\sqrt{c+2U}>0 \quad \mbox{\rm for} \;\; z \in (0,\pi)
$$
subject to the boundary condition $U'(0)=0$. By separation of variables we can solve this equation uniquely
to find that $U(z) =\tfrac{1}{6}(z^2 -3c)$. In view of the condition $U(\pi)=c$, this implies that
$c=\frac{\pi^2}{9}=c^*$ which proves the claim.
\medskip

\begin{figure}
\center
\includegraphics[scale=0.45]{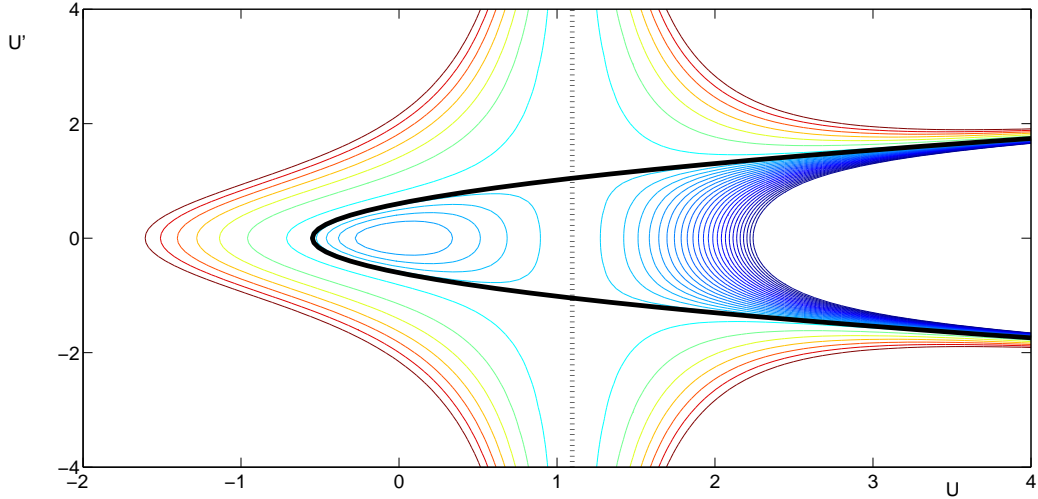}
\caption{Phase plane portrait obtained from level curves of the first-order invariant (\ref{energy-ODE}) for some $c > 0$.
The dashed black line indicates the
singularity line $U = c$.
The solid black curve to the left of the singular line corresponds to the parabolic profile (\ref{peaked-wave}). }
\label{fig-phase-plane}
\end{figure}

$\bullet$  Let us now analyze the situation where there exists $z_1 \in (0,\pi)$
such that $U(\pm z_1) = c$ and equation (\ref{contradiction-ODE}) holds separately for
$ z \in (0,z_1)$ and for $z\in(z_1,\pi)$. Let us assume that $U(z) < c$ for $z \in (0,z_1)$. If $U(z) > c$ for $z \in (0,z_1)$, the proof is analogous. There are two possibilities, either $(\partial_z^{-1} U)(\pm z_1) = 0$ or $(\partial_z^{-1} U)(\pm z_1) \neq 0$.\smallskip

If $(\partial_z^{-1} U)(\pm z_1) = 0$, then by the same argument as above the first-order invariant $E$
is continuous and constant on $[-z_1,z_1]$ with $E |_{z = \pm z_1} = E_{z_1}$.

For $z \in (z_1,\pi]$  the solution corresponding to the level set $E = E_{z_1}$  can either be continued uniquely from the region with $U(z)<c$  into the region with $U(z) > c$ , or $z=z_1$ represents a turning point (a local maximum for $U$) and the solution can be continued uniquely into the region with $U(z) < c$ for $z \gtrsim z_1$\footnote{The notation $z\gtrsim z_1$ means that $0<z-z_1<\varepsilon$ for some small $\varepsilon >0$, and equivalently for the reverse inequality. }. However, $U(z)>0$ for  $c>0$  and $U(z)<0$ for $c<0$ in the interval $z\in(z_1,\pi]$. Therefore, the first variant implies that $(\partial_z^{-1} U)(z) =  \int_0^z U(z')dz' \neq 0$ for every $z \in (z_1,\pi]$ which contradicts $(\partial_z^{-1} U)(\pi) =0$. The second continuation is possible
but does not belong to the class of single-lobe periodic waves, see Remark \ref{remark-non-uniqueness}.\smallskip

If $(\partial_z^{-1} U)(\pm z_1) \neq 0$, then the contradiction arises from the fact
that, since  $(\partial_z^{-1}U)(z)$ is continuous at $z_1$ and equation (\ref{contradiction-ODE})
holds separately for $ z \in (0,z_1)$ and for $z\in(z_1,\pi)$,  the change of the sign of $U'(z)$ across $z_1$ is determined
by the change of the sign of $c - U(z)$ across $z_1$.  Indeed, if $U(z) < c$ both for $z \lesssim z_1$ and $z \gtrsim z_1$,
then it follows from (\ref{contradiction-ODE}) that the sign of $U'(z)$ remains the same for  $ z \in (0,z_1)$ and $z\in(z_1,\pi)$.  But
this is impossible since $U'(z)$ must change sign for $z \lesssim z_1$ and $z \gtrsim z_1$
if $U(z) < c$ on both sides of $z_1$. If on the other hand $U(z) < c$ for $z \lesssim z_1$ and $U(z) > c$ for $z \gtrsim z_1$,
then it follows again from (\ref{contradiction-ODE}) that the sign of $U'(z)$ changes, in contradiction with
the monotone increase of $U(z)$  for all  $z \in (0,\pi)$.
Hence, both possibilities with $(\partial_z^{-1} U)(\pm z_1) \neq 0$  yield a contradiction.
\medskip

Combing all these arguments we find that the only single-lobe peaked periodic wave has
parabolic profile (\ref{peaked-wave}), which is Lipschitz at the peak  $U(\pm \pi) = c$.
\end{proof}

\begin{remark}
There is a simple way to obtain other peaked periodic waves in the boundary-value problem (\ref{ODE}).
One can flip the periodic wave with parabolic profile at
a point $z_0 \in (0,\pi)$ and pack two such waves over one period. This possibility is allowed in the proof
of Lemma \ref{lemma-no}, but not in the class of single-lobe periodic waves.
Similarly, one can pack three and more periods of the peaked wave with parabolic
profiles. Definition \ref{def-single-lobe} eliminates this type of non-uniqueness of the peaked
periodic wave in the boundary-value problem (\ref{ODE}).
\label{remark-non-uniqueness}
\end{remark}

\begin{remark}
\label{remark-small_amp_peakons}
With the following formal transformation
\begin{equation}
\label{coor-transf}
U(z) = {\bf u}(\zeta),\quad    z = \int_0^{\zeta} (c - {\bf u}(\zeta')) d \zeta',
\end{equation}
smooth periodic waves with profile $U$ satisfying the quasilinear second-order equation\footnote{The quasi-linear
equation (\ref{ode-1}) is a derivative of the first equation in the boundary-value problem (\ref{ODE})
in $z$, which is justified for smooth periodic waves in Lemma \ref{lemma-smooth-wave}.}
\begin{equation}
\label{ode-1}
\frac{d}{dz} \left[ c - U(z) \right] \frac{dU}{dz} + U = 0
\end{equation}
are related to smooth periodic waves with profile ${\bf u}$ satisfying the semi-linear second-order equation
\begin{equation}
\label{ode-2}
\frac{d^2 {\bf u}}{d \zeta^2} + (c - {\bf u}) {\bf u} = 0.
\end{equation}
Although all periodic solutions of (\ref{ode-2}) are smooth, the coordinate transformation
(\ref{coor-transf}) fails to be invertible if ${\bf u}(\zeta) = c$ for some $\zeta$. Such points
generate singularities in the periodic solutions of the quasi-linear equation (\ref{ode-1}) since
$$
U'(z) = \frac{{\bf u}'(\zeta)}{c - {\bf u}(\zeta)}.
$$
In \cite{Hakkaev1}, small-amplitude peaked periodic
waves of (\ref{ode-1}) with $c < 0$ were constructed from small-amplitude smooth periodic waves
of (\ref{ode-2}) with a coordinate transformation similar to (\ref{coor-transf}).
The corresponding profile $U$ for such peaked periodic waves has a square root singularity of the form
\begin{equation}
\label{wrong-wave}
    U(z)=  c + \O(\sqrt{\pi^2-z^2})   \quad \text{ as  } \quad z\to \pm \pi.
\end{equation}
Our analysis in the proof of Lemma \ref{lemma-no} shows\footnote{Solutions of \cite{Hakkaev1} have
nonzero mean value, hence Lemma \ref{lemma-no} does not apply directly. However, the arguments in the proof lead
to the same conclusion also for solutions with nonzero mean. Indeed, if $U$ has a non-zero mean,
$\partial_z^{-1} U(z)$ may not be zero at $z = \pm \pi$. However, if we translate  the solution by half a period
so that the singularity is placed at $z=0$, then $\partial_z^{-1} U(0) =0$ by oddness of $\partial_z^{-1} U$ and
we can use the same contradiction as the one obtained from \eqref{contradiction-ODE}.}
that such solutions cannot exist, since the expansion (\ref{wrong-wave}) implies that $U \in C^{1/2}_{\rm per}$.
We conclude that the small-amplitude peaked periodic waves constructed in \cite{Hakkaev1} are artefacts of
the construction method and do not satisfy the boundary-value problem (\ref{ODE}).
\end{remark}

\begin{remark}
\label{rem-loop}
In  \cite{Step} several different types of travelling wave solutions to the reduced Ostrovsky equation
were constructed by means of phase plane analysis. Three types of solitary waves (see Fig.~9 in~\cite{Step}) were found for $c<0$.
One of them is a loop soliton, given by a multi-valued function, which  is studied in many publications \cite{loop,vakh1,vakh2}. The other two solutions have points of infinite slope (cusps), either at the maximum or at the inflection points.
The cusped solitary waves were also constructed in \cite{Stefanov} by using the transformation (\ref{coor-transf}).
By using similar arguments as in the proof of Lemma \ref{lemma-no}, the existence of the cusped waves as weak solutions
to the reduced Ostrovsky equation can be disproved.
\end{remark}

\section{Linear instability of the peaked periodic wave}
\label{sec-3}

We add a {\em co-periodic} perturbation $v$ to the travelling wave $U$, that is,
a perturbation with the same period $2\pi$. Truncating the quadratic terms
and moving with the reference frame of the travelling wave yields
the linearized evolution problem in the form
\begin{equation}
\label{linOst}
\left\{ \begin{array}{l} v_t + \partial_z \left[ (U(z)-c) v \right]  = \partial_z^{-1} v, \quad t > 0,\\
v|_{t=0} = v_0.
\end{array} \right.
\end{equation}
The linearized evolution equation can be formulated in the form
$v_t = \partial_z L v$ defined by the self-adjoint operator
\begin{equation}
\label{operator-L}
L = P_0 \left( \partial_z^{-2} + c - U(z) \right) P_0 : \; \dot{L}^2_{\rm per} \to \dot{L}^2_{\rm per},
\end{equation}
where $P_0 : L^2_{\rm per} \to \dot{L}^2_{\rm per}$ is the projection operator
that removes the mean value of $2\pi$-periodic functions. The form
$v_t = \partial_z L v$ is related to the formulation of the reduced Ostrovsky equation
in the travelling wave coordinate $z = x-ct$ as a Hamiltonian system defined by
the symplectic operator $\partial_z$ and the conserved energy function
$H_c(u) = H(u) + c Q(u)$, where
\begin{equation}
\label{energy-nonlinear}
H(u) = \int_{-\pi}^{\pi} \left[ -(\partial_z^{-1} u)^2 - \frac{1}{3} u^3 \right] dz, \quad
Q(u) = \int_{-\pi}^{\pi} u^2 dz
\end{equation}
are the conserved energy and momentum functionals for the reduced Ostrovsky equation (\ref{redOst}).
The periodic wave $u = U$ is a critical point of $H_c(u)$ and the self-adjoint operator $L$ is the Hessian operator
of the energy function $H_c(u)$ at the periodic wave $u = U$.

Thanks to the translational invariance
of the boundary-value problem (\ref{ODE}), $L \partial_z U = 0$, where
$\partial_z U \in \dot{L}^2_{\rm per}$, holds for both the smooth periodic waves
of Lemma \ref{lemma-smooth-wave} and the peaked periodic wave (\ref{peaked-wave})
in Lemma \ref{lemma-no}.
Associated to the translational
eigenvector is the symplectic orthogonality constraint
$\langle U, v \rangle = 0$. This constraint is used to study
both the evolution of the Cauchy problem (\ref{linOst}) and the spectrum of the linearized operator
\begin{equation}
\label{eq-Lz}
\partial_z L : X^1_{\rm per} \subset \dot{L}^2_{\rm per} \to \dot{L}^2_{\rm per},
\end{equation}
where $X^1_{\rm per} = \{ v \in \dot{L}^2_{\rm per} : \; (c - U) v \in H^1_{\rm per} \}$ is the maximal domain
of $\partial_z L$. See \cite{Bronski,Haragus,Pel}.

\begin{remark}
\label{rem-embedding}
For smooth periodic waves we have $c - U(z) > 0$ for every $z \in [-\pi,\pi]$ so that
$X^1_{\rm per} \equiv \dot{H}^1_{\rm per}$. For the peaked periodic wave $U_*$ with speed $c_*$, the space
$\dot{H}^1_{\rm per}$ is continuously embedded into $X^1_{\rm per}$ since $U_*$ is bounded, but
$\dot{H}^1_{\rm per}$ is not equivalent to $X^1_{\rm per}$.  Indeed, if a perturbation $v$ to $U_*$
is piecewise $C^1_{\rm per}$ with a finite jump-discontinuity at $z = \pm \pi$, then
$v \notin \dot H^1_{per}$ but $v\in X^1_{\rm per}$ since $(c_* - U_*) v\in\dot H^1_{\rm per}$
in view of the fact that $U_*(\pm \pi) =c_*$.
\end{remark}

In what follows, $\langle \cdot , \cdot \rangle$ and $\| v \|_{L^2_{\rm per}}$ denote the inner product
and the $L^2$ norm with integration over $ [-\pi,\pi]$, respectively. In the case of the
peaked periodic wave $U_*$ with the speed $c_*$, we equip $X^1_{\rm per}$ with the norm
\begin{equation}
\label{norm-X-1}
\| v \|_{X^1_{\rm per}} := \| v \|_{L^2_{\rm per}} + \| \partial_z\, [(c_* - U_*) v\,] \|_{L^2_{\rm per}}.
\end{equation}
We distinguish two concepts of stability of $2 \pi$-periodic waves with respect to linearization.

\begin{definition}
\label{def-spectral-stability}
The travelling wave $U$ is said to be spectrally stable
if $\sigma(\partial_z L) \subset i \mathbb{R}$ in $\dot{L}^2_{\rm per}$.
Otherwise, it is said to be spectrally unstable.
\end{definition}

\begin{definition}
\label{def-linear-stability}
The travelling wave $U$ is said to be linearly stable
if for every $v_0 \in X^1_{\rm per}$ satisfying
$\langle U, v_0 \rangle = 0$, there exists $C > 0$ and
a unique global solution $v \in C(\mathbb{R},X^1_{\rm per})$
to the Cauchy problem (\ref{linOst}) such that
\begin{equation}
\label{bounds-semigroup}
\| v(t) \|_{X^1_{\rm per}} \leq C \| v_0 \|_{X^1_{\rm per}}, \quad t > 0.
\end{equation}
Otherwise, it is said to be linearly unstable.
\end{definition}

In \cite{GP17}, we have proved that the smooth periodic waves of Lemma \ref{lemma-smooth-wave}
are spectrally stable in the sense of Definition \ref{def-spectral-stability}.
Here we intend to show that the peaked periodic wave $U_*$ of Lemma \ref{lemma-no} given in  (\ref{peaked-wave}) is linearly unstable in the sense of Definition \ref{def-linear-stability}.
The linear instability is due to the sharp exponential growth
of the unique global solution to the Cauchy problem (\ref{linOst}) with $U = U_*$ :
\begin{equation}
\label{growth-semigroup}
C \| v_0 \|_{L^2_{\rm per}} e^{\pi t/6} \leq \| v(t) \|_{L^2_{\rm per}}
\leq \| v_0 \|_{L^2_{\rm per}} e^{\pi t/6}, \quad t > 0,
\end{equation}
for some $C \in (0,1)$. We will obtain these bounds in two steps.
In the first step,  carried out in Section \ref{section-trunc}, we apply
the method of characteristics to the truncated linearized equation (\ref{linOst})
without the dispersive term $\partial_z^{-1} v$ and obtain the sharp
bounds (\ref{growth-semigroup}) for all initial conditions $v_0 \in X^1_{\rm per}$
satisfying the constraint
\begin{equation}
\label{constraint-1}
\int_{-\pi}^{\pi} z v_0(z)^2 dz = 0.
\end{equation}
In the second step, carried out in Section \ref{section-full},
we will show that the bounds (\ref{growth-semigroup}) remain true in the full linearized
equation (\ref{linOst}) for a subset of initial conditions $v_0 \in X^1_{\rm per}$
satisfying the constraint (\ref{constraint-1}) and the additional constraint
\begin{equation}
\label{constraint-2}
\int_{-\pi}^{\pi} z^2 v_0(z) dz = 0,
\end{equation}
which arises due to the orthogonality condition $\langle U, v \rangle = 0$ in Definition \ref{def-linear-stability}
and the zero-mean condition on $v_0$.
Regarding spectral stability or instability of the peaked periodic wave (\ref{peaked-wave}),
we will show in Section \ref{section-L} that $\sigma(L)$ in $\dot{L}^2_{\rm per}$ is given by
a continuous spectrum on $[0,\pi^2/6]$, which includes
the embedded eigenvalue $\lambda_0 = 0$ with the eigenvector $\partial_z U$,
and a simple negative eigenvalue $\lambda_1 < 0$. As a result, no spectral gap
appears between $\lambda_0 = 0$ and the continuous spectrum, hence it is impossible
to solve the spectral stability problem by applying the standard methods from \cite{Bronski,Haragus,Pel}.

\subsection{Linear instability of truncated evolution}
\label{section-trunc}

For the peaked periodic wave (\ref{peaked-wave}), we obtain the simple
expression
\begin{equation}
\label{explicit-U-c}
U_*(z) - c_* = \frac{1}{6} (z^2-\pi^2), \quad z \in [-\pi,\pi].
\end{equation}
By removing the term $\partial_z^{-1} v$ from the linearized evolution problem (\ref{linOst}) and using the explicit expression (\ref{explicit-U-c}), we can write the truncated evolution problem
in the form
\begin{equation}
\label{truncOst}
\left\{ \begin{array}{l} v_t + \frac{1}{6} \partial_z \left[ (z^2 - \pi^2) v \right]  = 0, \quad t > 0,\\
v|_{t=0} = v_0, \end{array} \right.
\end{equation}
where the initial data $v_0$ is taken in $X^1_{\rm per}$. The evolution problem
can be solved by the method of characteristics along the family of characteristic
curves $z = Z(s,t)$, where $s \in [-\pi,\pi]$ is a parameter for the initial data
and $t \geq 0$ is the evolution time. Defining
\begin{equation}
\label{curves}
\left\{ \begin{array}{l}
\frac{d}{dt} Z(s,t) = \frac{1}{6} \left[ Z(s,t)^2 - \pi^2 \right], \quad t > 0,\\
Z(s,0) = s, \end{array} \right.
\end{equation}
and setting $V(s,t) := v(Z(s,t),t)$ yields the evolution problem
in the form
\begin{equation}
\label{evolution}
\left\{ \begin{array}{l}
\frac{d}{dt} V(s,t) = -\frac{1}{3} Z(s,t) V(s,t), \quad t > 0,\\
V(s,0) = v_0(s). \end{array} \right.
\end{equation}

The family of characteristic curves is obtained by integrating the differential
equation (\ref{curves}) with the parameter $s \in [-\pi,\pi]$. Because $Z = \pm \pi$
are critical points of the differential equation (\ref{curves}), the family of characteristic curves
remain inside the invariant region $[-\pi,\pi]$ for every $t \geq 0$. The family
of characteristic curves can be obtained in the explicit form
\begin{equation}
\label{curves-explicit}
Z(s,t) = \pi \frac{s \cosh(\pi t/6) - \pi \sinh(\pi t/6)}{\pi \cosh(\pi t/6) - s \sinh(\pi t/6)}, \quad s \in [-\pi,\pi], \;\; t \in \mathbb{R}.
\end{equation}
For later use of the chain rule we compute
\begin{equation}
\label{curves-derivative}
e^{\frac{1}{3} \int_0^t Z(s,t') dt'} = \frac{\partial}{\partial s} Z(s,t) =
\frac{\pi^2}{[\pi \cosh(\pi t/6) - s \sinh(\pi t/6)]^2}, \quad s \in [-\pi,\pi], \;\; t \in \mathbb{R}.
\end{equation}

The explicit solution for $V$ in characteristic variables is obtained by integrating the differential
equation (\ref{evolution}) with respect to the parameter $s \in [-\pi,\pi]$:
\begin{equation*}
V(s,t) = v_0(s) e^{-\frac{1}{3} \int_0^t Z(s,t') dt'}.
\end{equation*}
In view of (\ref{curves-derivative}), the explicit solution is given by
\begin{equation}
\label{evolution-explicit}
V(s,t) = \frac{1}{\pi^2} [\pi \cosh(\pi t/6) - s \sinh(\pi t/6)]^2 v_0(s), \quad
s \in [-\pi,\pi], \;\; t \in \mathbb{R}.
\end{equation}

\begin{remark}
\label{rem-invariance}
Since $Z(\pm \pi,t) = \pm \pi$ for every $t \in \mathbb{R}$, we have $V(\pm \pi,t) = e^{\mp \pi t/3} v(\pm \pi,t)$,
hence $V(-\pi,t) = V(\pi,t)$ for $t \neq 0$
if and only if $v_0(\pm \pi) = 0$, in which case $V(\pm \pi,t) = 0$ for every $t \in \mathbb{R}$.
Therefore, $V(\cdot,t) \notin \dot{H}^1_{\rm per}$ for $t \neq 0$ if
$v_0 \in \dot{H}^1_{\rm per}$ with $v_0(\pm \pi) \neq 0$.

\end{remark}

By using the explicit solutions (\ref{curves-explicit}) and (\ref{evolution-explicit}), we
are able to state and prove the following linear instability result for the truncated evolution problem
(\ref{truncOst}).

 \begin{lemma}
\label{lemma-truncated}
For every $v_0 \in X^1_{\rm per}$, there exists a unique global
solution $v \in C(\mathbb{R},X^1_{\rm per})$
to the Cauchy problem (\ref{truncOst}) satisfying the upper bound
\begin{equation}
\label{growth-truncated-upper}
\| v(t) \|_{L^2_{\rm per}} \leq \| v_0 \|_{L^2_{\rm per}} e^{\pi t/6}, \quad t > 0.
\end{equation}
If $\int_{-\pi}^{\pi} s v_0(s)^2 ds = 0$, then the global solution satisfies the lower bound
\begin{equation}
\label{growth-truncated}
\frac{1}{2} \| v_0 \|_{L^2_{\rm per}} e^{\pi t/6} \leq \| v(t) \|_{L^2_{\rm per}}, \quad t > 0.
\end{equation}
\end{lemma}

\begin{proof}
Existence of a global solution in the explicit form (\ref{curves-explicit}) and (\ref{evolution-explicit})
is obtained from the method of characteristics. By using the chain rule and (\ref{curves-derivative}), we verify
that the mean-zero constraint is preserved by the time evolution:
\begin{eqnarray*}
\int_{-\pi}^{\pi} v(z,t) dz = \int_{-\pi}^{\pi} V(s,t) \frac{\partial Z}{\partial s} ds
= \int_{-\pi}^{\pi} v_0(s) ds = 0.
\end{eqnarray*}
The explicit expression (\ref{evolution-explicit}) implies that $V(\cdot,t) \in X^1_{\rm per}$
if $v_0 \in X^1_{\rm per}$ and $t \in \mathbb{R}$. On the other hand, the explicit expression
(\ref{curves-explicit}) implies that for every $\tau > 0$, there exists $C_{\tau} > 0$ such that
$$
\frac{\partial}{\partial s} Z(s,t) \geq C_{\tau}, \quad s \in [-\pi,\pi], \;\; t \in [-\tau,\tau].
$$
Hence, the chain rule implies that $v(\cdot,t) \in X^1_{\rm per}$ if $v_0 \in X^1_{\rm per}$
and $t \in \mathbb{R}$. Uniqueness of such global solutions follows by standard theory (see Theorem 3.1 in \cite{Bressan}).

It remains to prove the sharp exponential growth in the bounds (\ref{growth-truncated-upper}) and (\ref{growth-truncated}).
By the chain rule, we obtain
\begin{eqnarray*}
\int_{-\pi}^{\pi} v(z,t)^2 dz & = & \int_{-\pi}^{\pi} V(s,t)^2 \frac{\partial Z}{\partial s} ds
= \frac{1}{\pi^2}  \int_{-\pi}^{\pi} [\pi \cosh(\pi t/6) - s \sinh(\pi t/6)]^2 v_0(s)^2 ds.
\end{eqnarray*}
From here, we have the upper bound
$$
\| v(t) \|^2_{L^2_{\rm per}} \leq e^{\pi t/3} \| v_0 \|^2_{L^2_{\rm per}}
$$
and the lower bound under the additional condition $\int_{-\pi}^{\pi} s v_0(s)^2 ds = 0$:
$$
\| v(t) \|^2_{L^2_{\rm per}} = \cosh(\pi t/6)^2 \| v_0 \|^2_{L^2_{\rm per}}
+ \frac{1}{\pi^2} \sinh(\pi t/6)^2 \| s v_0 \|^2_{L^2_{\rm per}} \geq \frac{1}{4} e^{\pi t/3} \| v_0 \|^2_{L^2_{\rm per}}.
$$
Taking the square root of these bounds yields (\ref{growth-truncated-upper}) and (\ref{growth-truncated}).
\end{proof}

\begin{remark}
By the chain rule, we also have
\begin{eqnarray*}
\int_{-\pi}^{\pi} \left[ \partial_z (\pi^2 - z^2) v(z,t) \right]^2 dz
= \frac{1}{\pi^2}  \int_{-\pi}^{\pi} [\pi \cosh(\pi t/6) - s \sinh(\pi t/6)]^2 \left[ \partial_s (\pi^2 - s^2) v_0(s) \right]^2 ds,
\end{eqnarray*}
from which the sharp exponential growth with the same growth rate as in the bounds (\ref{growth-truncated-upper}) and
(\ref{growth-truncated}) can be established for the second term in the
$X^1_{\rm per}$ norm given by (\ref{norm-X-1}).
\end{remark}

\begin{remark}
\label{remark-L1}
The global solution in Lemma \ref{lemma-truncated} remains bounded in $L^1$.
This follows from the chain rule:
\begin{eqnarray*}
\int_{-\pi}^{\pi} |v(z,t)| dz = \int_{-\pi}^{\pi} |V(s,t)| \frac{\partial Z}{\partial s} ds
= \int_{-\pi}^{\pi} |v_0(s)| ds.
\end{eqnarray*}
Since
\begin{equation}
\label{L1-L2}
\| v_0 \|_{L^1_{\rm per}} \leq (2\pi)^{1/2} \| v_0 \|_{L^2_{\rm per}},
\end{equation}
hence $v_0 \in \dot{L}^2_{\rm per}$ implies $v_0 \in L^1$.
Extending this bound to the time-dependent solution,
\begin{equation}
\label{L1-L2-time}
\| v(t) \|_{L^1_{\rm per}} \leq (2\pi)^{1/2} \| v(t) \|_{L^2_{\rm per}}, \quad t > 0,
\end{equation}
shows that the $L^1$ norm of the global solution $v(t)$ may remain bounded
even if the $L^2$ norm of this solution grows exponentially.
\end{remark}

\begin{remark}
\label{remark-energy}
Truncating a quadratic form associated with the self-adjoint operator $L$ in (\ref{operator-L})
and using the chain rule yield the energy conservation for the truncated evolution (\ref{truncOst}):
\begin{eqnarray*}
\int_{-\pi}^{\pi} (\pi^2 - z^2) v(z,t)^2 dz = \int_{-\pi}^{\pi} \left[ \pi^2 - Z(s,t)^2 \right] V(s,t)^2 \frac{\partial Z}{\partial s} ds
= \int_{-\pi}^{\pi} (\pi^2 - s^2) v_0(s)^2 ds.
\end{eqnarray*}
The energy conservation shows that
the truncated evolution leads to the exponential growth of $\| v(t) \|_{L^2_{\rm per}}^2$
and $\| z v(t) \|_{L^2_{\rm per}}^2$ but the difference between the two squared norms remains bounded.
\end{remark}

\begin{remark}
For the smooth periodic waves of Lemma \ref{lemma-smooth-wave} satisfying $U(z) < c$ for every $z \in [-\pi,\pi]$,
the truncated energy $\int_{-\pi}^{\pi}(c-U) v^2 dz$ is coercive in the $L^2$ norm, hence the energy conservation
$$
\int_{-\pi}^{\pi} \left[ c - U(z)  \right] v(z,t)^2 dz = \int_{-\pi}^{\pi} \left[ c-U(z) \right] v_0(s)^2 ds
$$
implies a global time-independent bound on $\| v(t) \|_{L^2_{\rm per}}^2$,
where $v(t)$ is a solution of the truncation of the linear evolution
equation (\ref{linOst}) without the $\partial_z^{-1} v$ term.
\end{remark}

\begin{remark}
For the smooth periodic waves of Lemma \ref{lemma-smooth-wave}, the characteristic curves reach the boundaries $z = \pm \pi$
in finite time because $z = \pm \pi$ are not critical points of the differential equations
for the characteristic curves. On the other hand, for the peaked periodic wave (\ref{peaked-wave}), the characteristic curves
reach  the boundaries $z = \pm \pi$ in infinite time. The latter property induces exponential growth of the global solutions to
the Cauchy problem (\ref{truncOst}) as is shown in Lemma \ref{lemma-truncated}.
\end{remark}

\subsection{Linear instability of full evolution}
\label{section-full}

Here we consider the full linearized evolution problem (\ref{linOst})
with (\ref{explicit-U-c}) and rewrite the evolution problem in the form
\begin{equation}
\label{fullOst}
\left\{ \begin{array}{l} v_t + \frac{1}{6} \partial_z \left[ (z^2 - \pi^2) v \right]  = \partial_z^{-1} v, \quad t > 0,\\
v|_{t=0} = v_0, \end{array} \right.
\end{equation}
where the initial data $v_0$ is taken in $X^1_{\rm per}$.

\begin{lemma}
\label{lemma-sg}
For every $v_0\in X^1_{\rm per}$ there exists a unique global solution $v \in C(\mathbb{R},X^1_{\rm per})$
of the Cauchy problem  \eqref{fullOst}.
\end{lemma}

\begin{proof}
By Lemma \ref{lemma-truncated}, the Cauchy problem (\ref{truncOst}) with $v_0 \in X^1_{\rm per}$
has a unique global solution $v \in C(\mathbb{R},X^1_{\rm per})$. In the framework of semigroup theory,
the evolution equation  \eqref{truncOst} can be written in the form $v_t = A_0 v$,
where
$$
A_0 := \frac{1}{6} \partial_z  (\pi^2 - z^2) v.
$$
Existence of a unique global solution $v \in C(\mathbb{R},X^1_{\rm per})$ implies
that the operator $A_0$ with domain $D(A_0) = X^1_{\rm per}$ is the infinitesimal generator of
a strongly continuous semigroup $(S_0(t))_{t\geq 0}$ on $\dot L^2_{\rm per}$.
Since  $\partial_z^{-1} : \dot{L}^2_{\rm per} \to \dot{L}^2_{\rm per}$ is a bounded operator,
the Bounded Perturbation Theorem (see Theorem III,1.3 on p.~158 in \cite{Engel}) implies that
the operator
$$
A := A_0 + \partial_z^{-1}
$$
with the same domain $D(A) = D(A_0) = X^1_{\rm per}$
also generates a strongly continuous semigroup $(S(t))_{t\geq 0}$ on $\dot{L}^2_{\rm per}$.
Therefore, the evolution equation in the Cauchy problem \eqref{fullOst} can be viewed
as a bounded perturbation of the evolution equation in the Cauchy problem \eqref{truncOst}.
The assertion of the Lemma then follows by Proposition 6.2 on  p.~145 in \cite{Engel}.
\end{proof}

In what follows, we obtain bounds on the global solution $v \in C(\mathbb{R},X^1_{\rm per})$
to the Cauchy problem (\ref{fullOst}). First, we note the following upper bound on the growth of the global solution.

\begin{lemma}
\label{lemma-upper-bound}
A global solution $v \in C(\mathbb{R},X^1_{\rm per})$
to the Cauchy problem (\ref{fullOst}) in Lemma \ref{lemma-sg} satisfies the upper bound
\begin{equation}
\label{growth-full}
\| v(t) \|_{L^2_{\rm per}} \leq \| v_0 \|_{L^2_{\rm per}} e^{\pi t/6}, \quad t > 0.
\end{equation}
\end{lemma}

\begin{proof}
Note the following integration yields
$$
\int_{-\pi}^{\pi} v (\partial_z^{-1} v )dz = \frac{1}{2} (\partial_z^{-1} v)^2 |_{z = -\pi}^{z = \pi} = 0,
$$
since $\partial_z^{-1} v \in H^1_{\rm per}$ and
hence $\partial_z^{-1} v \in C_{\rm per}$ by Sobolev's embedding.
Integrating by parts yields the following balance equation
$$
\frac{d}{dt} \frac{1}{2} \|v(t) \|^2_{L^2_{\rm per}} = \frac{1}{6} \int_{-\pi}^{\pi} v \partial_z \left[ (\pi^2 - z^2) v \right] dz
= - \frac{1}{6} \int_{-\pi}^{\pi} (\pi^2 - z^2) v \partial_z v dz = -\frac{1}{6} \int_{-\pi}^{\pi} z v^2 dz.
$$
Hence
$$
\frac{d}{dt} \|v(t) \|^2_{L^2_{\rm per}} \leq \frac{\pi}{3} \|v(t) \|^2_{L^2_{\rm per}}
$$
and Gronwall's inequality yields the desired bound (\ref{growth-full}).
\end{proof}

In order to obtain the lower bound on the $L^2$ norm of the global solution
to the Cauchy problem (\ref{fullOst}), we  use the generalized
method of characteristics and treat $\partial_z^{-1} v(z,t)$
as a source term in (\ref{truncOst}). This term satisfies the following useful bound (also proven in \cite{LPS}).

\begin{lemma}
If $g := \partial_z^{-1} v \in \dot{H}^1_{\rm per}$, then
\begin{equation}
\label{L1-Linfty}
\| g\|_{L^{\infty}_{\rm per}} \leq \| v \|_{L^1_{\rm per}}.
\end{equation}
\end{lemma}

\begin{proof}
By Sobolev embedding of $H^1_{\rm per}$ into $C_{\rm per}$,
$g$ is a continuous $2\pi$-periodic function with zero mean. Therefore, there exists $\zeta \in [-\pi,\pi]$ such that
$g(\zeta) = 0$. For every $z \in [-\pi,\pi]$, we can write
$$
g(z) = \int_{\zeta}^z v(z') dz',
$$
from which bound (\ref{L1-Linfty}) follows. Note that $L^2$ is continuously embedded into
$L^1$ because of the bound (\ref{L1-L2}).
\end{proof}

By using the family of characteristic curves $z = Z(s,t)$ with $s \in [-\pi,\pi]$ and $t \geq 0$, where
$Z$ is defined by the same initial-value problem (\ref{curves}), and setting $V(s,t) := v(Z(s,t),t)$
and $G(s,t) := g(Z(s,t),t)$,
we obtain the evolution problem in the form
\begin{equation}
\label{evolution-full}
\left\{ \begin{array}{l}
\frac{d}{dt} V(s,t) = -\frac{1}{3} Z(s,t) V(s,t) + G(s,t), \quad t > 0,\\
V(s,0) = v_0(s). \end{array} \right.
\end{equation}

The family of characteristic curves $Z$ is still given by the same explicit form (\ref{curves-explicit}).
Integrating the differential equation (\ref{evolution-full}) with an integrating factor yields
the explicit solution for $V$ in the form
\begin{equation}
\label{evolution-explicit-full}
V(s,t) = \left[ v_0(s) + \int_0^t G(s,t') e^{\frac{1}{3} \int_0^{t'} Z(s,t'') dt''} dt' \right] e^{-\frac{1}{3} \int_0^t Z(s,t') dt'}
\end{equation}
By using the explicit solution (\ref{evolution-explicit-full}), we
are able to prove the linear instability result for the Cauchy problem (\ref{fullOst}).

\begin{lemma}
\label{lemma-full}
There exists $v_0 \in X^1_{\rm per}$ and $C > 0$ such that the unique global
solution $v \in C(\mathbb{R},X^1_{\rm per})$
to the Cauchy problem (\ref{fullOst}) in Lemma \ref{lemma-sg} satisfies the lower bound
\begin{equation}
\label{growth-full-lower}
\| v(t) \|_{L^2_{\rm per}} \geq C \| v_0 \|_{L^2_{\rm per}} e^{\pi t/6}, \quad t > 0.
\end{equation}
\end{lemma}

\begin{proof}
By the chain rule, the explicit expression (\ref{evolution-explicit-full})  with the help of (\ref{curves-derivative})
yields the following equation:
\begin{eqnarray*}
\int_{-\pi}^{\pi} v(z,t)^2 dz & = & \int_{-\pi}^{\pi} V(s,t)^2 \frac{\partial Z}{\partial s} ds \\
& = & \frac{1}{\pi^2}
\int_{-\pi}^{\pi} [\pi \cosh(\pi t/6) - s \sinh(\pi t/6)]^2 \\
& \phantom{t} & \phantom{text}
\times \left[ v_0(s)
+ \int_0^t \frac{\pi^2 G(s,t')}{[\pi \cosh(\pi t'/6) - s \sinh(\pi t'/6)]^2} dt' \right]^2 ds.
\end{eqnarray*}
Let us assume the same constraint $\int_{-\pi}^{\pi} s v_0(s)^2 ds = 0$ as in Lemma \ref{lemma-truncated}.
Neglecting positive terms in the lower bound, we obtain
\begin{eqnarray}
\label{lower-bound-interim}
\| v(t) \|^2_{L^2_{\rm per}} & \geq & \frac{1}{4} e^{\pi t/3} \| v_0 \|^2_{L^2_{\rm per}}\\
\nonumber
& \phantom{t} &
- 2 \int_{-\pi}^{\pi} \int_0^t |v_0(s)| |G(s,t')| \frac{[\pi \cosh(\pi t/6) - s \sinh(\pi t/6)]^2}{[\pi \cosh(\pi t'/6) - s \sinh(\pi t'/6)]^2} dt' ds.
\end{eqnarray}
Let us define for any $t > 0$,
$$
K(t,t',s) := \frac{\pi \cosh(\pi t/6) - s \sinh(\pi t/6)}{\pi \cosh(\pi t'/6) - s \sinh(\pi t'/6)}, \quad t' \in [0,t], \;\; s \in [-\pi,\pi].
$$
We prove that for every $0 \leq t' \leq t$,
\begin{equation}
\label{bound-on-K}
\sup_{s \in [-\pi,\pi]} K(t,t',s) = e^{\pi (t-t')/6}.
\end{equation}
Indeed, $K(t,t',s) = e^{\pi (t-t')/6} M(t,t',s)$, where
$$
M(t,t',s) := \frac{(\pi - s) + (\pi + s) e^{-\pi t/3}}{(\pi - s) + (\pi + s) e^{-\pi t'/3}},
$$
and $M$ is monotonically decreasing since $\partial_s M(t,t',s) \leq 0$ for every $t' \in [0,t]$
and $s \in [-\pi,\pi]$. Therefore, $M$ has a maximum at $s = -\pi$, where $M(t,t',-\pi) = 1$.

By using (\ref{lower-bound-interim}) and (\ref{bound-on-K}), we obtain
\begin{eqnarray*}
\| v(t) \|^2_{L^2_{\rm per}} & \geq & \frac{1}{4} e^{\pi t/3} \| v_0 \|^2_{L^2_{\rm per}}
- 2 \| v_0 \|_{L^1_{\rm per}} \int_0^t \| g(t') \|_{L^{\infty}_{\rm per}} e^{\pi (t-t')/3} dt'\\
& \geq & \frac{1}{4} e^{\pi t/3} \| v_0 \|^2_{L^2_{\rm per}}
- 2 \| v_0 \|_{L^1_{\rm per}} \int_0^t \| v(t') \|_{L^1_{\rm per}} e^{\pi (t-t')/3} dt'\\
& \geq & \frac{1}{4} e^{\pi t/3} \| v_0 \|^2_{L^2_{\rm per}}
- 2 \sqrt{2\pi} \| v_0 \|_{L^1_{\rm per}} \| v_0 \|_{L^2_{\rm per}} e^{\pi t/3} \int_0^t e^{-\pi t'/6} dt',
\end{eqnarray*}
where (\ref{L1-L2-time}), (\ref{growth-full}), and (\ref{L1-Linfty}) have been used in the last two inequalities.
Hence,
\begin{eqnarray*}
\| v(t) \|^2_{L^2_{\rm per}} e^{-\pi t/3} & \geq & \| v_0 \|_{L^2_{\rm per}}
 \left ( \frac{1}{4} \| v_0 \|_{L^2_{\rm per}}
- \frac{12\sqrt{2}}{\sqrt{\pi}} \| v_0 \|_{L^1_{\rm per}}\right )
\end{eqnarray*}
and since $\| v_0 \|_{L^2_{\rm per}}$ can be much larger than $\| v_0 \|_{L^1_{\rm per}}$ by the bound (\ref{L1-L2}),
there exist $v_0 \in X^1_{\rm per}$ and $C^2 \in (0,1/4)$ such that
\begin{equation}
\label{inequality}
\| v_0 \|_{L^1_{\rm per}} \leq \frac{\sqrt{\pi} (1-4C^2)}{48 \sqrt{2}} \| v_0 \|_{L^2_{\rm per}},
\end{equation}
and hence
\begin{eqnarray}
\label{bounds-desired-squared}
\| v(t) \|^2_{L^2_{\rm per}} e^{-\pi t/3} & \geq & C^2 \| v_0 \|^2_{L^2_{\rm per}}.
\end{eqnarray}
This yields the desired bound (\ref{growth-full-lower}).
\end{proof}

\begin{remark}
\label{rem-contraints}
Let us show that there exist functions $v_0 \in X^1_{\rm per}$ satisfying
the constraints (\ref{constraint-1}), (\ref{constraint-2}), and (\ref{inequality}).
Indeed, if $v_0$ is odd, then $v_0^2$ is even,
hence the two constraints (\ref{constraint-1}) and (\ref{constraint-2})
are satisfied simultaneously. From the class of
odd initial data we need to pick functions in $X^1_{\rm per}$
that satisfy the inequality (\ref{inequality}) for a fixed $C^2 \in (0,1/4)$. For example,
we can consider the following odd function in $\dot{H}^1_{\rm per}\subset X^1_{\rm per}$
\begin{equation}
\label{example-v-0}
v_0(x) = \frac{x (\pi^2-x^2)}{1 + a^2 x^2}, \quad x \in [-\pi,\pi],
\end{equation}
where $a > 0$ is a parameter. We obtain by direct computation,
$$
\| v_0 \|_{L^1_{\rm per}} = \left( \frac{\pi^2}{a^2} + \frac{1}{a^4} \right) \log(1+\pi^2 a^2) - \frac{\pi^2}{a^2}
$$
and
$$
\| v_0 \|_{L^2_{\rm per}}^2 = \frac{1}{a^3} \left[ \left( \pi^4 + \frac{6\pi^2}{a^2} + \frac{5}{a^4}\right) \arctan(\pi a)
 - \frac{\pi (15 + 13 \pi^2 a^2)}{3 a^3} \right].
$$
Since $\| v_0 \|_{L^1_{\rm per}} = \mathcal{O}(\log(a) a^{-2})$ decays to zero as $a \to \infty$ faster than
$\| v_0 \|_{L^2_{\rm per}} = \mathcal{O}(a^{-3/2})$,
inequality (\ref{inequality}) can be satisfied for sufficiently large $a$.
\end{remark}

\begin{remark}
If $v_0(\pm \pi) = 0$ like in the example (\ref{example-v-0}), then $v_0 \in \dot{H}^1_{\rm per}$
and the truncated linearized evolution (\ref{truncOst}) preserves the constraint
$v(\pm \pi,t) = 0$ for every $t \in \mathbb{R}$, see Remark \ref{rem-invariance}. However,
the integral term $\partial_z^{-1} v$ in the full linearized evolution (\ref{fullOst}) does not
generally preserve the same constraint because it is uniquely defined from the condition
that $\partial_z^{-1} v$ has zero mean. As a result, the full linearized equation
does not generally admit a solution $v \in C(\mathbb{R},\dot{H}^1_{\rm per})$ even if $v_0 \in \dot{H}^1_{\rm per}$.
\end{remark}

\begin{remark}
In the presence of the source term $G$, we are not able to show that $\| v(t) \|_{L^1_{\rm per}}$
remains bounded as $t \to \infty$, see Remark \ref{remark-L1}. By using the integral
$$
\int_{-\pi}^{\pi} \frac{\pi^2}{[\pi \cosh(\pi t'/6) - s \sinh(\pi t'/6)]^2} ds = 2\pi, \quad t' \in [0,t],
$$
we obtain the bound
\begin{eqnarray*}
\| v(t) \|_{L^1_{\rm per}} & \leq & \| v_0 \|_{L^1_{\rm per}} + 2\pi \int_0^t \| g(t') \|_{L^{\infty}_{\rm per}} dt',
\end{eqnarray*}
in view of \eqref{curves-derivative} and \eqref{evolution-explicit-full}.
Thanks to the bound (\ref{L1-Linfty}), the inequality is closed as follows:
\begin{eqnarray*}
\| v(t) \|_{L^1_{\rm per}}
& \leq & \| v_0 \|_{L^1_{\rm per}} + 2\pi \int_0^t \| v(t') \|_{L^1_{\rm per}} dt'.
\end{eqnarray*}
By Gronwall's inequality, this bound gives the fast exponential growth
$$
\| v(t) \|_{L^1_{\rm per}} \leq  \| v_0 \|_{L^1_{\rm per}} e^{2 \pi t},
$$
which cannot be sharp because $\| v(t) \|_{L^1_{\rm per}}$ is bounded by a slowly growing exponential
function that follows from the bounds (\ref{L1-L2-time}) and (\ref{growth-full}).
\end{remark}

\begin{remark}
There exists a conserved energy for the Cauchy problem (\ref{fullOst}),
see Remark \ref{remark-energy}, which is given by
\begin{equation}
\label{conserved-energy-lin}
\langle L v(t), v(t) \rangle =
\langle L v_0, v_0 \rangle,
\end{equation}
where the self-adjoint operator $L$ is defined by (\ref{operator-L}).
However, the conserved quantity (\ref{conserved-energy-lin}) does not
prevent $\| v(t) \|_{L^2_{\rm per}}$ from growing exponentially fast as $t \to \infty$
because the bounded operator $L$ is not coercive under the constraint (\ref{constraint-2}),
see Lemma \ref{lemma-spectrum}.
\end{remark}

\subsection{Spectrum of the linear self-adjoint operator $L$}
\label{section-L}

Here we consider the spectrum $\sigma(L)$ of the linear self-adjoint operator $L$ defined by (\ref{operator-L}).
We will prove that $\sigma(L)$ consists of the continuous spectrum on $[0,\pi^2/6]$, which includes
the embedded eigenvalue $\lambda_0 = 0$ with the eigenvector $\partial_z U$,
and a simple negative eigenvalue $\lambda_1 < 0$. No spectral gap
appears between $\lambda_0 = 0$ and the continuous spectrum.
The following lemma gives the corresponding result.

\begin{lemma}
\label{lemma-spectrum}
The spectrum of the self-adjoint operator $L$ given by (\ref{operator-L}) is
\begin{equation}
\label{spectrum}
\sigma(L) = \{ \lambda_1 \} \cup \left[ 0, \frac{\pi^2}{6} \right],
\end{equation}
where $\lambda_1 < 0$ is the unique zero of the transcendental equation
\begin{equation}
\label{transc}
(\pi^2 + 3 \lambda) \log \frac{\sqrt{\pi^2 - 6 \lambda} + \pi}{\sqrt{\pi^2 - 6 \lambda} - \pi} - 3 \pi \sqrt{\pi^2 - 6 \lambda} = 0,
\quad \lambda < 0.
\end{equation}
\end{lemma}

\begin{proof}
By the spectral theorem (see, e.g., Definition 8.39, Theorem 8.70, and Theorem 8.71 in \cite{RR}),
the spectrum of the self-adjoint operator $L$ in $\dot{L}^2_{\rm per}$ denoted by $\sigma(L)$
may consist of only two disjoint sets on the real line:
the point spectrum of eigenvalues with eigenvectors in $\dot{L}^2_{\rm per}$
denoted by $\sigma_p(L)$ and the continuous spectrum denoted by $\sigma_c(L)$,
where the resolvent operator exists but is unbounded.

The self-adjoint operator $L$ in (\ref{operator-L}) is given by the sum of a bounded
operator $L_0$ and a compact operator $K$ given by
\begin{equation}
\label{bounded}
L_0 := \frac{1}{6} P_0 \left( \pi^2 - z^2 \right) P_0 : \; \dot{L}^2_{\rm per} \to \dot{L}^2_{\rm per}
\end{equation}
and
\begin{equation}
\label{compact}
K := P_0 \partial_z^{-2} P_0 : \; \dot{L}^2_{\rm per} \to \dot{L}^2_{\rm per}.
\end{equation}
Moreover, the compact operator is in the trace class since $\sum_{n=1}^{\infty} n^{-2} < \infty$.
By Kato's Theorem \cite{Kato} (see Theorem~4.4 on p.~540 in \cite{Kato-text}),
$\sigma_c(L) = \sigma_c(L_0)$. We show that $[0,\pi^2/6] \subseteq \sigma_c(L_0)$
by considering the odd functions in $\dot{L}^2_{\rm per}$, which can be represented
by the Fourier sine series. Let us denote the space of odd functions in $\dot{L}^2_{\rm per}$
by $L^2_{\rm per, odd}$. Then,
$$
L_0 f = \frac{1}{6} (\pi^2 - z^2) f, \quad \forall f \in L^2_{\rm per, odd}.
$$
Then, $\sigma_c(L_0)$ in $L^2_{\rm per, odd}$
coincides with the range of the multiplicative function
$h(z) = \frac{1}{6}(\pi^2-z^2)$ for $z \in [-\pi,\pi]$, which is $[0,\pi^2/6]$. Hence,
$[0,\pi^2/6] \subseteq \sigma_c(L_0)$ in $\dot{L}^2_{\rm per}$.

Let us show that $[0,\pi^2/6] \equiv \sigma_c(L_0)$ by working with the resolvent equation
$(L_0 -\lambda I) f = g$ for given $g \in \dot{L}^2_{\rm per}$ and $\lambda \notin [0,\pi^2/6]$.
The resolvent equation can be written in the component form for $z \in [-\pi,\pi]$:
$$
\frac{1}{6} (\pi^2 - 6 \lambda - z^2) f(z) - k(f) = g(z), \quad k(f) := \frac{1}{12 \pi} \int_{-\pi}^{\pi} (\pi^2-z^2) f(z) dz,
$$
where $f \in \dot{L}^2_{\rm per}$ is supposed to satisfy the zero-mean constraint
$\int_{-\pi}^{\pi} f(z) dz  = 0$. Computing the solution explicitly,
$$
f(z) = \frac{6}{\pi^2 - 6 \lambda - z^2} \left[ g(z) + k(f) \right],
$$
and using the zero mean constraint, we can define $k(f)$ in terms of $g$:
$$
k(f) = \frac{\int_{-\pi}^{\pi} \frac{g(z)}{\pi^2 - 6 \lambda - z^2} dz}{\int_{-\pi}^{\pi} \frac{1}{\pi^2 - 6 \lambda - z^2} dz}.
$$
For every $\lambda \notin [0,\pi^2/6]$, there exist positive constants $C_{\lambda}, C_{\lambda}' > 0$
such that
$$
\sup_{z \in [-\pi,\pi]} \frac{6}{|\pi^2 - 6 \lambda - z^2|} \leq C_{\lambda}, \quad
\left| \int_{-\pi}^{\pi} \frac{6}{\pi^2 - 6 \lambda - z^2} dz \right| \geq C_{\lambda}'.
$$
As a result, we obtain the bound
$$
\| f \|_{L^2_{\rm per}} \leq C_{\lambda} \left[ \| g \|_{L^2_{\rm per}} + |k(f)| \sqrt{2\pi} \right]
\leq C_{\lambda} \left[ 1 + 2\pi (C_{\lambda}')^{-1} C_{\lambda} \right] \| g \|_{L^2_{\rm per}}.
$$
Therefore, the resolvent operator $(L_0 - \lambda I)^{-1} : \dot{L}^2_{\rm per} \to \dot{L}^2_{\rm per}$
is bounded for every $\lambda \notin [0,\pi^2/6]$ so that $\sigma_c(L_0) = [0,\pi^2/6]$.\medskip

In order to study $\sigma_p(L) \in \mathbb{R} \backslash [0,\pi^2/6]$,
we consider the spectral problem for operator $L$ with the spectral parameter $\lambda \notin [0,\pi^2/6]$:
\begin{equation}
\label{spectral-problem}
\frac{1}{6} P_0 \left( \pi^2 - z^2 \right) w + P_0 \partial_z^{-2} w = \lambda w, \quad w \in \dot{L}^2_{\rm per}.
\end{equation}
Since $\partial_z^{-2} w \in H^2_{\rm per}$, bootstrapping arguments show that $w \in H^2_{\rm loc}$
on any compact subset in $(-\pi,\pi)$. Iterations of bootstrapping arguments yield $w \in H^{\infty}_{\rm loc}$.
Therefore, the spectral problem (\ref{spectral-problem}) can be differentiated twice on a compact subset in $(-\pi,\pi)$,
after which it is rewritten as the second-order differential equation
\begin{equation}
\label{second-order-ODE}
\left( \pi^2 - z^2 - 6 \lambda \right) \frac{d^2 w}{d z^2} - 4 z \frac{dw}{dz} + 4 w(z) = 0, \quad w \in H^{\infty}_{\rm loc},
\end{equation}
with the two linearly independent solutions for $\lambda \in \mathbb{R} \backslash [0,\pi^2/6]$,
\begin{eqnarray*}
w_1(z) = z
\end{eqnarray*}
and
\begin{eqnarray*}
w_2(z) = \left\{ \begin{array}{ll} -1 + \frac{z^2}{2(\pi^2 - z^2 - 6 \lambda)} + \frac{3z}{4\sqrt{\pi^2 - 6 \lambda}}
\log \frac{\sqrt{\pi^2-6\lambda} + z}{\sqrt{\pi^2 - 6 \lambda} - z}, \quad & \lambda < 0,\\
-1 + \frac{z^2}{2(\pi^2 - z^2 - 6 \lambda)} - \frac{3z}{2\sqrt{6 \lambda-\pi^2}}
\arctan \frac{z}{\sqrt{6 \lambda - \pi^2}}, \quad & \lambda > \frac{\pi^2}{6}. \end{array} \right.
\end{eqnarray*}
The first solution corresponds to the eigenvector $\partial_z U$ of the spectral problem (\ref{spectral-problem})
for the eigenvalue $\lambda_0 = 0$, which is embedded into $\sigma_c(L) = [0,\pi^2/6]$.
Since eigenvectors of the self-adjoint operator for distinct eigenvalues are orthogonal,
we are looking for solutions $w$ of the spectral problem (\ref{spectral-problem})
such that $\langle w, w_1 \rangle = 0$. Therefore, we take\footnote{Note that $\langle w_2,w_1 \rangle = 0$
because $w_1$ is odd and $w_2$ is even.} $w = w_2$
and extend it from $H^2_{\rm loc}$ to $\dot{L}^2_{\rm per}$.
This extension is achieved if and only if $w$ has zero mean, that is,
\begin{equation}
\label{zero-mean}
0 = \frac{1}{2\pi} \int_{-\pi}^{\pi} w_2(z) dz = \left\{ \begin{array}{ll}
-\frac{3}{4} + \frac{\pi^2 + 3 \lambda}{4 \pi \sqrt{\pi^2 - 6 \lambda}}
\log \frac{\sqrt{\pi^2 - 6 \lambda} + \pi}{\sqrt{\pi^2 - 6 \lambda} - \pi} \quad & \lambda < 0, \\
-\frac{3}{4} - \frac{\pi^2 + 3 \lambda}{2 \pi \sqrt{6 \lambda - \pi^2}}
\arctan \frac{\pi}{\sqrt{6 \lambda - \pi^2}} \quad & \lambda > \frac{\pi^2}{6}. \end{array} \right.
\end{equation}
The piecewise graph of the right-hand side of the zero-mean constraint
(\ref{zero-mean}) on $(-\infty,0)$ and $(\pi^2/6,\infty)$
is shown on Figure \ref{fig-root}. The first line of the zero-mean constraint (\ref{zero-mean}) is equivalent
to the transcendental equation (\ref{transc}) and it has
only one simple zero at $\lambda_1 \approx -0.2262$. The second line of (\ref{zero-mean})
does not have any zeros. Hence, $\lambda_1 < 0$ is the only eigenvalue in $\sigma_p(L)$.
\end{proof}

\begin{figure}[htbp]
\center
\includegraphics[scale=0.5]{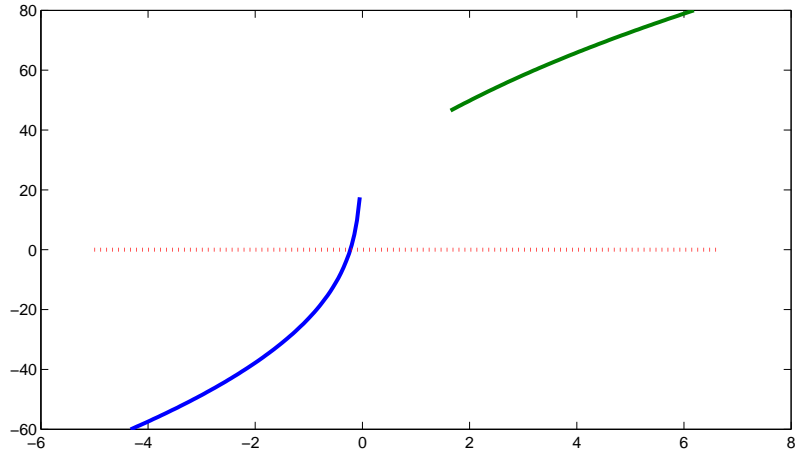}
\caption{The graph of the right-hand side of the zero-mean constraint (\ref{zero-mean})
on $(-\infty,0)$ and $(\pi^2/6,\infty)$ as a function of the spectral parameter $\lambda$.
Only one simple zero $\lambda_1<0$ exists.}
\label{fig-root}
\end{figure}

\begin{remark}
\label{rem-spectrum}
For the smooth periodic waves of Lemma \ref{lemma-smooth-wave},
we proved in \cite{GP17} that $\sigma(L)$ in $\dot{L}^2_{\rm per}$ includes a simple negative eigenvalue,
a simple zero eigenvalue with the eigenvector $\partial_z U$, and the rest of the spectrum is positive and bounded away
from zero. Hence, the spectral gap is present in the case of smooth periodic waves.
This enabled us   in \cite{GP17} to use Hamilton-Krein index theory to deduce that $\sigma(\partial_z L) \subset i \mathbb{R}$
and hence to deduce spectral stability of the smooth periodic waves according to Definition \ref{def-spectral-stability}.
By the standard analysis involving the conserved quantity (\ref{conserved-energy-lin}), see \cite{Haragus-Li},
this spectral stability result transfers to linear stability of the smooth periodic waves
according to Definition \ref{def-linear-stability}. In the spectral problem for the peaked periodic wave,
however, this  spectral gap is not present. Therefore, we are not able to deduce spectral instability
of the peaked periodic wave from the spectrum of $L$.
\end{remark}

\section{Discussion}
\label{sect-discussion}

We have studied peaked periodic traveling wave solutions of the reduced Ostrovsky equation \eqref{redOst}.
We found that the peaked periodic wave with parabolic shape $U_*$ is the unique periodic traveling wave
with a single minimum per period and that the boundary-value problem (\ref{ODE}) does not admit H\"{o}lder continuous solutions,
see Lemma \ref{lemma-no}. As a consequence, existence of cusped waves obtained in \cite{Step} as well as existence
of small-amplitude peaked waves found in \cite{Hakkaev1} was disproven.

Furthermore, we proved that the peaked periodic wave $U_*$  is linearly unstable with respect to
co-periodic perturbations in the space $X^1_{\rm per}$, which is the maximal domain of the linearized operator $\partial_z L$,
see Lemma \ref{lemma-full}. This result was obtained using sharp exponential bounds on the $L^2$ norm of perturbations $v$ of $U_*$ in $X^1_{\rm per}$
satisfying the Cauchy problem \eqref{linOst} with the peaked periodic wave $U_*$ and for the wave speed $c_*$.

Passing from linear to nonlinear instability is often a delicate issue. Several authors have shown that linear instability
directly implies nonlinear instability if a part of the spectrum of the linearized operator
is located in the right half of the complex plane, see for instance \cite{Friedlander1997,Shatah} and
Theorem 5.1.5 in \cite{Henry1981}. However, these approaches do not work for the reduced Ostrovsky equation
since the linearized evolution is defined in the space $X^1_{\rm per}$ whereas the nonlinear evolution is
defined in $\dot H^s_{\rm per}$ with $s > 3/2$. It is not clear if the local well-posedness results can be extended
to the space $X^1_{\rm per}$. It is also unclear how the peaks of the peaked periodic wave move under the flow
of the reduced Ostrovsky equation  \eqref{redOst}. For these reasons, nonlinear instability of the peaked periodic wave
in the reduced Ostrovsky equation remains an open problem for now.

\bigskip

{\bf Acknowledgements.} This project was initiated during the research program on
Nonlinear Water Waves at Isaac Newton Institute at Cambridge in August 2017.
Both authors thank Mats Ehrnstr\"{o}m for useful discussions during the program and afterward.
Computations in the proof of Lemma \ref{lemma-spectrum} were performed back in 2015
in collaboration with Ted Johnson (UCL).
DEP acknowledges a financial support from the State task program in the sphere
of scientific activity of Ministry of Education and Science of the Russian Federation
(Task No. 5.5176.2017/8.9) and from the grant of President of Russian Federation
for the leading scientific schools (NSH-2685.2018.5). Finally, the authors thank
the referees for useful remarks which helped to improve the manuscript.

\end{document}